\documentclass[a4,12pt]{amsart}
\usepackage{amssymb,amstext,amsmath,amscd,amsthm,amsfonts,enumerate,latexsym}
\usepackage{color}
\usepackage{graphicx}
\usepackage[all]{xy}
\theoremstyle{plain}
\newtheorem{thm}{Theorem}[section]

\newtheorem*{thm*}{Theorem}
\newtheorem*{cor*}{Corollary}

\newtheorem{prop}[thm]{Proposition}
\newtheorem{proposition}[thm]{Proposition}

\newtheorem{lem}[thm]{Lemma}
\newtheorem{cor}[thm]{Corollary}

\newtheorem*{claim*}{Claim}

\theoremstyle{definition}

\newtheorem{ex}[thm]{Example}

\newtheorem{rem}[thm]{Remark}
\newtheorem{remark}[thm]{Remark}

\newtheorem{ques}[thm]{Question}

\theoremstyle{remark}

\numberwithin{equation}{thm}

\def\Z{\mathbb{Z}}

\def\Ker{\operatorname{Ker}}

\def\m{\mathfrak m}

\def\N{\Bbb N}

\newcommand{\Ap}{\mathrm{Ap}}

\newcommand{\rmF}{\mathrm{F}}
\newcommand{\rmI}{\mathrm{I}}

\def\gr{\mbox{\rm gr}}

\title[Frobenius numbers and the first Hilbert coefficients]{Frobenius numbers and the first Hilbert coefficients of certain numerical semigroup rings}

\author{Do Van Kien}
\address{Department of Mathematics, Hanoi Pedagogical University 2, 32 Nguyen Van Linh street, Xuan Hoa, Phu Tho, Vietnam}
\email{dovankien@hpu2.edu.vn}

\author{Pham Hung Quy}
\address{Department of Mathematics, FPT University, Hanoi, Vietnam}
\email{quyph@fe.edu.vn}
\begin{document}
\maketitle
\begin{abstract}
Let $a,b$ be positive integers. In this note, we study the numerical semigroup $H=\left<a,a+1,b\right>$ and and the associated numerical semigroup ring $R=k[[H]]$. Under the certain conditions, we provide explicit formulas for  the Frobenius number of $H$ and for the first Hilbert coefficient of $R$.
\end{abstract}
\section{Introduction}\label{section1}
Throughout this note, let $n_1,n_2,\ldots,n_e$ be positive integers with $\gcd(n_1,n_2,\ldots,n_e)=1$. Let $$H=\left<n_1,n_2,\ldots,n_e\right>=\{c_1n_1+c_2n_2+\cdots+c_en_e\mid c_i\in\Z_{\ge 0},\, \forall i=1,2,\ldots,e\}$$ be a numerical semigroup. We always assume $n_1=\min\{n_1,n_2,\ldots,n_e\}$ and that $H$ is minimally generated by $\{n_i\}_{i=1}^e$. We call $n_1$ and $e$ the {\it multiplicity} and the {\it embedding dimension} of $H$, respectively. Since  $\gcd(n_1,n_2,\ldots,n_e)=1$, the complement $\N\setminus H$ is finite.  The largest number of  $\N\setminus H$ is said to be the {\it Frobenius number} of $H$ and is denoted by $\rmF(H)$. We also define the {\it genus} of $H$ by $g(H)=|\N\setminus H|$. The problem of determining explicit formulas for $\rmF(H)$ in terms of the generators
$n_i$ is a classical and important problem in number theory.  When the embedding dimension $e=2$, it is well known that $\rmF(H)=n_1n_2-n_1-n_2$. However, for $e\ge 3$ the problem remains open in general.

The first purpose of this article is to give the explicit formula for the Frobenius number of $H$ when $H$ is minimally generated by $\{a,a+1,a+d\}$.

One important approach to studying the combinatorial properties of $H$ is through its numerical semigroup ring. Let $R=k[[H]]=k[[t^{n_1},t^{n_2},\ldots,t^{n_e}]]\subseteq k[[t]]$ is the numerical semigroup ring associated to $H$ over a field $k$. This is a noetherian local domain with maximal ideal $\m=(t^{n_1},t^{n_2},\ldots,t^{n_e})$. Since $R$ is a one-dimensional Cohen-Macaulay local ring, there are integers $e_0(R)$ and $e_1(R)$ such that
$$\ell_R(R/\m^{n+1})=e_0(R){n+1\choose 1}-e_1(R)$$
for all sufficiently large integers $n$.
We call $e_0(R), e_1(R)$ the {\it multiplicity} and the {\it first Hilbert coefficient} of $R$, respectively. While it is well-known that $e_0(R)=n_1$, there is currently no general formula for computing $e_1(R)$ in terms of generators of $H$. This naturally leads to the following questions.
\begin{ques}\label{Chern}
	Find an explicit formula for the first Hilbert coefficient of $R$?
\end{ques}
In \cite{K75}, Kirby established bounds for $e_1(R)$. He proved that 
\begin{equation*}
	n_1-1\le e_1(R)\le \frac{1}{2}n_1(n_1-1).
\end{equation*}

The second purpose of this article is to give an explicit formula for the first Hilbert coefficient of the numerical semigroup ring associated to $H=\left<a,a+1,a+d\right>$. For the two main purposes of this article, let us set $n_1<n_2<n_3$ be positive integers such that $\gcd(n_1,n_2,n_3)=1$. Let $$H=\left<n_1,n_2,n_3\right>:=\{c_1n_1+c_2n_2+c_3n_3\mid c_1,c_2,c_3\in \N\}$$ be a numerical semigroup. Let $R=k[[H]]$ is the numerical semigroup ring associated to $H$ and $S = k[[x,y,z]]$ be the formal series ring. We consider the   $k$-algebra homomorphism $\varphi :S \to R$ is given by  $\varphi(x)=t^{n_1}, \varphi(y)=t^{n_2}, \varphi(z)=t^{n_3}$. Let $I= \Ker\varphi $ and call it the {\it defining ideal} of $R$. It is well known that if $H$ is not symmetric, then the ideal $I$ is minimally generated by 
\begin{equation*}
	\tag{$\sharp$} x^{\alpha+\alpha'}-y^{\beta'}z^\gamma, y^{\beta+\beta'}-x^{\alpha}z^{\gamma'},z^{\gamma+\gamma'}-x^{\alpha'}y^{\beta}
\end{equation*}
for some positive integers $\alpha, \beta, \gamma, \alpha', \beta', \gamma'$(cf. \cite{He70}). 

In \cite{NNW12}, the authors gave the formula for the Frobenius number of $H$ in terms of the parameters $\alpha, \beta, \gamma, \alpha', \beta', \gamma'$.
\begin{prop}[see \cite{NNW12}]\label{Fro}{\,}
\begin{enumerate}
		\item[1)] If $\beta'n_2>\alpha n_1$ then $\rmF(H)=\beta' n_2+(\gamma+\gamma')n_3-(n_1+n_2+n_3)$.
		\item[2)] If $\beta'n_2<\alpha n_1$ then $\rmF(H)=\alpha n_1+(\gamma+\gamma')n_3-(n_1+n_2+n_3)$.
	\end{enumerate} 
\end{prop}
However, it is very difficult to compute six parameters $\alpha, \beta, \gamma, \alpha', \beta', \gamma'$ in terms of $n_1,n_2,n_3$. In general, no explicit formulas for these parameters are known, except under special conditions. For instance, when $n_1, n_2, n_3$ form an arithmetic sequence, these parameters can be computed explicitly, as shown in \cite[Theorem 1.1]{GSS11}. It is quite surprising that even if $n_1$ and $n_2$ are two consecutive integers then there is still no formula for determining these parameters. 
In this paper, we describe explicitly the parameters $\alpha, \beta, \gamma, \alpha', \beta', \gamma'$ when $n_1=a, n_2=a+1$ and $n_3=a+d$. The first main result of this paper is stated as follows, providing formulas for $\rmF(H)$ as well as $e_1(R)$.
\begin{thm}[Propositions \ref{prop37}+\ref{prop38}+\ref{e1first}]
	For $a,d$ be positive integers with $1<d<a$. Let
	$H=\left<a,a+1,a+d\right>$ a numerical semigroup and $R=k[[H]]$ its numerical semigroup ring. Let $a=dq+r$ with $1\le r\le d-1$. Suppose  $a\ge d^2-3d$ then
\begin{enumerate}
	\item[1)] $\rmF(H)=\max\limits_{1<i<a}\left(i(a+1) -(d-1)a\left\lfloor \dfrac{i}{d}\right\rfloor\right)-a, \text{ and }$
	\item[2)] $e_1(R)=\dfrac{a(a-1)}{2}
	-(d-1)\left(d\dfrac{q(q-1)}{2}+rq\right).$
\end{enumerate}
In particular, for all $a\ge 7$, we have 
	\[
\rmF(a,a+1,a+5)=\begin{cases}
	\lfloor\dfrac{a}{5}\rfloor(a+5)+2a+3, & \text{if } a\equiv 4\pmod 5,\\[1ex]
	\lfloor\dfrac{a}{5}\rfloor(a+5)+2a-1, & \text{if } a\not\equiv 4\pmod 5.
\end{cases}
\]
\end{thm}

The second main result of this paper is stated as follows.
\begin{thm}[Theorem \ref{mainthm}+Proposition \ref{Fro-Chern}]
Let $a,d$ be positive integer with $a\ge 3, d\ge 2$ and $ \gcd(a,d)=1$. Let  $H=\left<a,a+1,a+d\right>$ a numerical semigroup and $R=k[[H]]$ its numerical semigroup ring. We write $a=dq+r \text{ with } 1\le r\le d-1.$  Suppose $q+r\ge d-2$ and $H$ is not symmetric, then the following assertions hold true.
\begin{enumerate}
	\item[1)] If $1\le r\le d-2$ then $\rmF(H)=\dfrac{1}{d}a^2+\dfrac{d^2-2d-r}{d}a-(r+1)$, and
	$$e_1(R)=\dfrac{1}{2d}(a^2+d(d-2)a-r(d-1)(d-r).$$
	\item[2)] If $r=d-1$ then $\rmF(H)=\dfrac{1}{d}a^2+\dfrac{d^2-3d+1}{d}a-1,$ and 
	 $$e_1(R)=\dfrac{1}{2d}(a^2+d(d-2)a-(d-1)^2).$$
\end{enumerate}
\end{thm}
The organization of this paper is as follows. In Section 2, we study the Ap\'{e}ry set and the Frobenius number of the numerical semigroup $H=\left<a,a+1,b\right>$. Section 3 is devoted to the investigation of the first Hilbert coefficient of the numerical semigroup ring associated to $H$. In particular, we provide explicit formulas for the Frobenius number of $H$ and for the first Hilbert coefficient in the case where $b=a+d$ with $d<a$.
\section{The Ap\'{e}ry set and the Frobenius number of $\left<a,a+1,b\right>$}
Let $H$ be a numerical semigroup and $h\in H, h\neq 0$. We put $\Ap(H,h)=\{m\in H\mid m-h\notin H\}$ and call it the {\it Ap\'{e}ry set} of $H$ with respect to $h$. By definition, it is obvious that $\rmF(H)=\max\left(\Ap(H,h)\right)-h$. For each $0\le i\le h-1$, let $\omega_i=\min\{m\in H\mid m\equiv i\pmod {h}\}$. The following give the Ap\'{e}ry set of $H$ with respect to $h$.
\begin{lem}[\cite{RGS09}, Lemma 2.4]\label{Aperyset}
	$\Ap(H,h)=\{0=\omega_0,\omega_1,\ldots,\omega_{h-1}\}$.
\end{lem} 

In this section, we investigate the first Hilbert coefficients and Frobenius numbers of numerical semigroups of embedding dimension $3$. Let $n_1<n_2<n_3$ be positive integers such that $\gcd(n_1,n_2,n_3)=1$. Let $H=\left<n_1,n_2,n_3\right>$ be a numerical semigroup. Let $R=k[[H]]$ is the numerical semigroup ring associated to $H$ and $e_1(R)$ is the first Hilbert coefficient of $R$. Let $S = k[[x,y,z]]$ be the formal series ring and $\varphi :S \to R$ be a $k$-algebra homomorphism given by  $\varphi(x)=t^{n_1}, \varphi(y)=t^{n_2}, \varphi(z)=t^{n_3}$. Let $I= \Ker\varphi $ is the defining ideal of $R$. For a given matrix $M$ with
entries in $S$, we denote by $\rmI_2(M)$ the ideal of $S$ generated by $2\times 2$ minors of $M$. It is known
that if $H$ is not symmetric, then the ideal $I$ is generated by the maximal minors of the Herzog matrix 
$$\begin{pmatrix}
	y^{\beta'}&z^{\gamma'}&x^{\alpha'}\\
	x^\alpha&y^\beta&z^\gamma
\end{pmatrix}$$
for some positive integers $\alpha, \beta, \gamma, \alpha', \beta', \gamma'$(cf. \cite{He70}). 

In this section, we consider the case where the numbers $n_1, n_2, n_3$ are not very different. Let $a, b$ be positive integers with $a+1<b$. We consider the numerical semigroup $H=\left<a,a+1,b\right>$ minimally generated by $\{a,a+1,b\}$.  We will explicitly calculate  the Frobenius number of $H$ and the first Hilbert coefficient of $R$. We write $b=aq+r$ for some positive integers $q$ and $1\le r\le a-1.$
\begin{lem}\label{ineq}
There are inequalities
 $q<r$ and $b<r(a+1)$.
\end{lem}
\begin{proof}
Suppose, on the contrary, that $q\ge r$. Then $q-r\ge 0$ and there are equalities $b=aq+r=(q-r)a+r(a+1)$. This contradicts to the minimality of the generating set $\{a,a+1,b\}$. Thus,  $q<r$. Hence, $aq+r<ar+r$ which implies $b<r(a+1)$.
\end{proof}
For each $0\le i\le a-1$, let $\omega_i=\min\{h\in H\mid h\equiv i\pmod {a}\}$. We know by Lemma \ref{Aperyset} that $\Ap(H,a)=\{0=\omega_0,\omega_1,\ldots,\omega_{a-1}\}$.

The following result have known recently by Robles-P\'{e}rez and Rosales.
\begin{prop}[see \cite{Rosal24}, Proposition 3.6] Suppose that $\omega_0<\omega_1<\cdots<\omega_{a-1}$. Then for all $0\le i\le a-1$, we have
$$\omega_i=i(a+1)-\lfloor\frac{i}{r}\rfloor(r(a+1)-b).$$
\end{prop}
\begin{ques}
If we remove the condition that $\{\omega_i\}_{0}^a$ is an increasing sequence, what is the formula for $\omega_i$?
\end{ques}
\begin{remark}
Obviously, $\omega_0=0$. Now fix each $1\le i\le a-1$. One has 
\begin{align*}
\omega_i&=\min\{xa+y(a+1)+zb\mid x,y,z\ge 0, xa+y(a+1)+zb\equiv i\pmod a\}\\
&=\min\{xa+y(a+1)+zb\mid x,y,z\ge 0, y+zr\equiv i\pmod a\}\\
&=\min\{y(a+1)+zb\mid 0\le y,z< a, y+zr\equiv i\pmod a\}.
\end{align*}


Then there are integers $0\le y_0,z_0<a$ satisfying $y_0+z_0r\equiv i\pmod a$ such that $\omega_i=y_0(a+1)+z_0b.$ Note that $i=y_0+z_0r-ka<a$. It implies that $k\ge 0$. Substituting $b=aq+r$, we obtain
$$\omega_i=i(a+1)+a\left(k(a+1)-z_0(r-q)\right).$$ Since $0\le y_0<a$, we get $$\frac{z_0r-i}{a}\le k=\frac{y_0+z_0r-i}{a}<\frac{z_0r-i}{a}+1.$$
Hence $k=\Big\lceil\dfrac{z_0r-i}{a}\Big\rceil$ so that
$$\omega_i=i(a+1)+a\left(\Big\lceil\dfrac{z_0r-i}{a}\Big\rceil(a+1)-z_0(r-q)\right).$$ 
\end{remark}
This observation leads to the following statement.
\begin{lem}\label{lem25} For each $1\le i\le a-1$, we have
$$\omega_i=i(a+1)+a.\min_{0\le z<a}F(z),$$
where $F(z)=\Big\lceil\dfrac{zr-i}{a}\Big\rceil(a+1)-z(r-q)$.
\end{lem}
To determine $\omega_i$, the following statement shows that 
we only need to compare $r$ values instead of $a$ values.
\begin{lem}\label{lem26} For each $1\le i\le a-1$, we have
$$\omega_i=i(a+1)+a.\min_{0\le j\le r}\{j(a+1)-(r-q)\lfloor\frac{ja+i}{r}\rfloor \}.$$
\end{lem}
\begin{proof}
We consider the function $F(z)=\Big\lceil\dfrac{zr-i}{a}\Big\rceil(a+1)-z(r-q)$ with $0\le z<a$. Observe that if $zr\le 0$ then $z\le \dfrac{i}{r}$ and  $\Big\lceil\dfrac{zr-i}{a}\Big\rceil=0$. Hence,
\begin{equation*}\tag{1}\min_{0\le z<a, \\ z\le i/5}F(z)=F(\lfloor\frac{i}{r}\rfloor)=-(r-q)\lfloor\frac{i}{r}\rfloor.
\end{equation*}
Otherwise, if $zr> 0$ then
$0<\dfrac{zr-i}{a}<r$ which implies that
$\Big\lceil\dfrac{zr-i}{a}\Big\rceil\in\{1,2,\ldots,r\}.$
On the other hands, for each $j\in\{1,2,\ldots,r\}$, we see that
$\Big\lceil\dfrac{zr-i}{a}\Big\rceil=j$ if and only if
$$\dfrac{zr-i}{a}\le j< \dfrac{zr-i}{a}+1.$$
Equivalently,
$$\dfrac{a(j-1)+i}{r}<z\le \dfrac{ja+i}{r}.$$
Therefore, 
\begin{equation*}\tag{2}\min_{0\le z<a, \\ \lceil\frac{zr-i}{a}\rceil=j}F(z)=j(a+1)+\min_{\frac{a(j-1)+i}{r}<z\le \frac{ja+i}{r}}(-z(r-q)). =j(a+1)-(r-q)\lfloor\frac{ja+i}{r}\rfloor.
\end{equation*}
Combining $(1)$ and $(2)$ we obtain
$$\omega_i=i(a+1)+a.\min_{0\le j\le r}\{j(a+1)-(r-q)\lfloor\frac{ja+i}{r}\rfloor \}.$$
\end{proof}

\begin{cor}
	For $0<q<a$ be integers and $H=\left<a,a+1,aq+q+1\right>$ a numerical semigroup. Then
	$$\rmF(H)=a^2-a-1-\lfloor\frac{a-1}{q+1}\rfloor.a$$
\end{cor}
\begin{proof}
	For all $1\le i,j\le a-1$, we see that
	$j(a+1)\ge \dfrac{ja+1}{q+1}\ge \lfloor\dfrac{ja+i}{q+1}\rfloor$. Hence, for each $1\le i\le a-1$, we have by Lemma \ref{lem26} that
	\begin{align*}
		\omega_i&=i(a+1)+a.\min_{0\le j\le q+1}\{j(a+1)-\lfloor\frac{ja+i}{q+1}\rfloor \}\\
		&=i(a+1)+a.\min\{-\lfloor\frac{i}{q+1}\rfloor, \min_{1\le j\le q+1}\{j(a+1)-\lfloor\frac{ja+i}{q+1}\rfloor \}\}\\
		&=i(a+1)-\lfloor\frac{i}{q+1}\rfloor.a.
	\end{align*}
	It is clear that the sequence $\{\omega_i\}_{i=1}^{a-1}$ is a strict increasing sequence. Therefore, $$\max_{1\le i\le a-1}\omega_i=\omega_{a-1}=a^2-1-\lfloor\frac{a-1}{q+1}\rfloor.a$$
	Consequently, 
	$$\rmF(H)=a^2-a-1-\lfloor\frac{a-1}{q+1}\rfloor.a$$
\end{proof}
\begin{proposition}\label{prop37}
	For $a,d$ be positive integers with $1<d<a$ and
	$H=\left<a,a+1,a+d\right>$ a numerical semigroup. Suppose  $a\ge d^2-3d$ then for each $0\le i\le a-1$, we have \[
	\omega_i = i(a+1) -(d-1)a\left\lfloor \frac{i}{d}\right\rfloor.\]
Consequently, $$\rmF(H)=\max_{1<i<a}\left(i(a+1) -(d-1)a\left\lfloor \frac{i}{d}\right\rfloor\right)-a$$.
\end{proposition}
\begin{proof}
	We have by Lemma \ref{lem26} that
	\begin{align*}
		\omega_i=i(a+1)+a.\min_{0\le j\le d}\{j(a+1)-4\lfloor\frac{ja+i}{d}\rfloor \}.
	\end{align*}
	We fix each $1\le j\le d$ and observe that
	\begin{align*}
		j(a+1)-(d-1).\lfloor\frac{ja+i}{d}\rfloor+(d-1)\lfloor\frac{i}{5}\rfloor
		&>j(a+1)-(d-1)\left(\dfrac{ja+i}{d}-\dfrac{i}{d}+1\right)\\
		&= \dfrac{j(a+d)-d^2+d}{d}.
	\end{align*}
	It yields
	\begin{align*}
		j(a+1)-(d-1).\lfloor\frac{ja+i}{d}\rfloor+(d-1)\lfloor\frac{i}{5}\rfloor
		&\ge\lfloor\dfrac{j(a+d)-d^2+d}{d}\rfloor+1\\
		&=\lfloor\dfrac{j(a+d)-d^2+d}{d}+1\rfloor\\
		&=\lfloor\dfrac{j(a+d)-d^2+2d}{d}\rfloor.
	\end{align*}
	Hence, if $a\ge d^2-3d$ then $\lfloor\dfrac{j(a+d)-d^2+2d}{d}\rfloor\ge 0$ which implies that $$	j(a+1)-(d-1).\lfloor\frac{ja+i}{d}\rfloor\ge -(d-1)\lfloor\frac{i}{d}\rfloor.$$ Therefore if $a\ge d^2-3d$ then
	\[
	\omega_i = i(a+1) -(d-1)a\left\lfloor \frac{i}{d}\right\rfloor .
	\]
\end{proof}
As a consequence of Proposition \ref{prop37}, we have the following statements.
\begin{proposition}\label{prop38}
	For $a\ge 7$ be an integer and
	$H=\left<a,a+1,a+5\right>$. Then for each $0\le i\le a-1$, we have 	\[
	\omega_i = i(a+1) - 4a\left\lfloor \frac{i}{5}\right\rfloor .\]
Moreover, 	\[
\rmF(H)=\begin{cases}
	\lfloor\dfrac{a}{5}\rfloor(a+5)+2a+3, & \text{if } a\equiv 4\pmod 5,\\[1ex]
	\lfloor\dfrac{a}{5}\rfloor(a+5)+2a-1, & \text{if } a\not\equiv 4\pmod 5.
\end{cases}
\]
\end{proposition}
\begin{proof}
If $a\ge 10$ then by Proposition \ref{prop37} we have that 
\[\omega_i = i(a+1) - 4a\left\lfloor\frac{i}{5}\right\rfloor,\]
for all $0\le i\le a-1$.

Moreover, a direct computation shows that this equality still holds true for each $7\le a\le 9$. This conclude the equalities for $\omega_i$'s as desired.

Now to determine the Frobenius of $H$ we need to find the largest value among $\omega_i$'s with $1\le i\le a-1$. Write $a=5m+s$ with $m>0, s\in\{0,1,2,3,4\}$. For each $1\le i\le a-1$,
we write $i=5k+r$ with $r\in\{0,1,2,3,4\}$. Then
\begin{align*}
	w_i=(5k+r)(a+1)-4ak
	&= k(a+5)+r(a+1)\\&=\dfrac{i-r}{5}(a+5)+r(a+1)\le \dfrac{a-1-r}{5}(a+5)+r(a+1).
\end{align*}
Note that $0\le r,s\le 4$, we get that
	\[
\left\lfloor\frac{a-1-r}{5}\right\rfloor=\left\lfloor m+\frac{s-1-r}{5}\right\rfloor=
	\begin{cases}
		m, & r\le s-1,\\
		m-1, & r\ge s.
	\end{cases}
	\]
Therefore, 	
\[\omega_i\le	\begin{cases}
	m(a+5)+r(a+1), & \text{ if }r\le s-1,\\
	(m-1)(a+5)+r(a+1), & \text{ if } r\ge s.
\end{cases}
\]
We consider the following two cases:\\
{\bf Case $s=4:$} In this case, we have for each $i=5k+r$ that
\[\omega_i\le	\begin{cases}
	m(a+5)+r(a+1), & \text{ if }r\le 3,\\
	(m-1)(a+5)+r(a+1), & \text{ if } r=4.
\end{cases}
\]
Hence for all $1\le i\le a-1$, we have $$\omega_i\le \max_{0\le j\le 3}\{m(a+5)+j(a+1),	(m-1)(a+5)+4(a+1)\}=m(a+5)+3(a+1).$$
The equality holds true if $i=5m+3=a-1$.\\
{\bf Case $0\le s\le 3:$} 
A similar argument as the $s=4$ case, we have for all $1\le i\le a-1$ that $$\omega_i\le (m-1)(a+5)+4(a+1).$$
	The equality holds true if $i=5(m-1)+4$.\\
	Therefore,
	\[
	\max_{0\le i\le a-1} w_i =
	\begin{cases}
		m(a+5)+3(a+1), & \text{ if } s=4,\\
		(m-1)(a+5)+4(a+1), & \text{ if } s\neq 4,
	\end{cases}
	\]
which implies that
$$\rmF(H)=	\max_{0\le i\le a-1} w_i-a=\begin{cases}
	\lfloor\dfrac{a}{5}\rfloor(a+5)+2a+3, & \text{if } a\equiv 4\pmod 5,\\[1ex]
	\lfloor\dfrac{a}{5}\rfloor(a+5)+2a-1, & \text{if } a\not\equiv 4\pmod 5.
\end{cases}$$
We therefore have completed the proof.
\end{proof}

In general, it is difficult to determine an explicit formula for $\rmF(H)$. However, under certain conditions, we can derive an explicit formula for it.

\begin{thm}\label{mainthm} Let $a,d$ be positive integers with $a\ge 3, d\ge 2$ and $ \gcd(a,d)=1$. Let  $H=\left<a,a+1,a+d\right>$ and $R=k[[H]]$ its numerical semigroup ring. Denote by $I$ the defining ideal of $R$. We write $a=dq+r \text{ with } 1\le r\le d-1.$  Suppose $q+r\ge d-2$, then
	$$I=\rmI_2\left( {\begin{array}{*{20}{c}}
			{{y^r}}&z&{{x^{q + r + 2 - d}}} \\ 
			{{x^{d - 1}}}&{{y^{d - r}}}&z^q 
	\end{array}} \right).$$
	In addition, if $H$ is not symmetric then the following assertions hold true.
	\begin{enumerate}
		\item[1)] If $1\le r\le d-2$ then $\rmF(H)=\dfrac{1}{d}a^2+\dfrac{d^2-2d-r}{d}a-(r+1)$.
		\item[2)] If $r=d-1$ then $\rmF(H)=\dfrac{1}{d}a^2+\dfrac{d^2-3d+1}{d}a-1.$
	\end{enumerate}
\end{thm}
\begin{proof}
	We consider the ideal 
	$$P=\rmI_2\left( {\begin{array}{*{20}{c}}
			{{y^r}}&z&{{x^{q + r + 2 - d}}} \\ 
			{{x^{d - 1}}}&{{y^{d - r}}}&z^q 
	\end{array}} \right)\subseteq S.$$
	By the definition of $\varphi$, it is clear that $P\subseteq I$. For the reverse conclusion, we observe that
	$$(x)+P=(x)+\rmI_2\left( {\begin{array}{*{20}{c}}
			{{y^r}}&z&0 \\ 
			0&{{y^{d - r}}}&z^q
	\end{array}}\right)=(x,y^d,y^rz^{q}, z^{q+1}).$$
	So, 
	\begin{align*}
		\ell_S\left(\dfrac{S}{(x)+P}\right)&=\dim_k\frac{k[[x,y,z]]}{(x,y^d,y^rz^{q}, z^{q+1})}\\
		&=\dim_k\frac{k[[y,z]]}{(y^d,y^rz^{q}, z^{q+1})}\\
		&=(q+1)r+q(d-r)\\
		&=qd+r\\
		&=a.
	\end{align*}
	Moreover, because $R\cong S/I$ we have $\ell_S\left(\dfrac{S}{(x)+I}\right)=\ell_R\left(\dfrac{R}{(t^a)}\right)=a.$
	This implies that  $$\ell_S\left(\frac{S}{(x)+P}\right)=\ell_S\left(\frac{S}{(x)+I}\right)$$ which yields $(x)+P=(x)+I$. Hence $I/P=x(I/P)$. Hence, by Nakayama's lemma (see the proof of $(2)\Rightarrow(4)$ in \cite[Theorem 2]{Ki20}), we get that $I=P$. Therefore
	\begin{equation*}
		\tag{*}	I=\rmI_2\left( {\begin{array}{*{20}{c}}
				y^r&z&x^{q + r + 2 - d}\\ 
				x^{d - 1}&y^{d - r}&z^{q} 
		\end{array}} \right).
	\end{equation*}

Now we assume that $H$ is not symmetric. We then have $q>0$ so that $a>d$.\\	{\bf Proof of 1) and 2)}. By the form of $I$ in $(*)$, the parameters in $(\sharp)$ are given by
	$$\alpha=d-1,\beta=d-r,\gamma=i,\alpha'=q+r+2-d,\beta'=r, \text{ and } \gamma'=1.$$
	
	If $r=d-1$ then $\beta'(a+1)=r(a+1)>(d-1)a=\alpha a$. Hence, by  Proposition \ref{Fro}, we have 
	\begin{align*}
		\rmF(H)&=\beta'(a+1) + (\gamma + \gamma')(a+d)-(a + a+1 + a+d)\\
		&=(d-1)(a+1)+(\frac{a-(d-1)}{d}+1)(a+d)-(3a+1+d)\\
		&=\frac{1}{d}a^2+\frac{d^2-3d+1}{d}a-1.
	\end{align*}
	
	If $1\le r\le d-2$ then since $a>d$, we have $\beta'(a+1)=r(a+1)\le (d-2)(a+1)<(d-1)a=\alpha a$. Hence, again by Proposition \ref{Fro} we get
	\begin{align*}
		\rmF(H)&=\alpha a + (\gamma + \gamma')(a+d)-(a + a+1 + a+d)\\
		&=(d-1)a+(\frac{a-r}{d}+1)(a+d)-(3a+1+d)\\
		&=\frac{1}{d}a^2+\frac{d^2-2d-r}{d}a-(r+1).
	\end{align*}
\end{proof}
\begin{cor}
Suppose that $H$ is not symmetric and one of the following conditions hold:
\begin{enumerate}
	\item[1)] $2\le d\le 4$.
	\item[2)]  $d\ge 5$ and $d-3\le r\le d-1$.
	\item[3)] $a>d^2-3d$.
\end{enumerate}
Then $\rmF(H)$ is given in Theorem \ref{mainthm}.
\end{cor}
\begin{proof}
Note that $q,r\ge 1$. It is easy to see that the condition $q+r\ge d-2$ is satisfied if one of the given conditions hold. The statement is followed. 
\end{proof}
We also see that if $a>d^2-3d$ then Theorem \ref{mainthm} and Proposition \ref{prop37} imply that $\omega_{a-1}=\max\{\omega_{i}\mid 0\le i\le a-1\}$.
\begin{rem}
	In remaining cases, that is, $d\ge 5$, $a=qd+r$ with $1\le q+r\le d-3$. We have not yet an explicit formula for $\rmF(H)$.
\end{rem}
In the case $q=r=1$ we have the follwing.
\begin{lem}\label{remaining} Let
	$H=\left<d+1,d+2,2d+1\right>$ with $d\ge 5$ and $I$ is the defining ideal of $R=k[[H]]$. Then the following assertions hold true.
	\begin{itemize}
		\item[1)] if $d\equiv 0\pmod 3$ then $I=\rmI_2\begin{pmatrix}
			y&z^{\frac{d}{3}}&x\\
			x^2&y^{\frac{2d}{3}}&z
		\end{pmatrix}$ and $\rmF(H)=\dfrac{2}{3}d^2+\dfrac{1}{3}d-1$.
		\item[2)] if $d\equiv 1\pmod 3$ then $I=\rmI_2\begin{pmatrix}
			y&z^{\frac{d+2}{3}}&x^3\\
			1&y^{\frac{2d-2}{3}}&z
		\end{pmatrix}$ and $\rmF(H)=\dfrac{2}{3}d^2+\dfrac{2}{3}d-\dfrac{1}{3}$.
		\item[3)] if $d\equiv 2\pmod 3$ then $I=\rmI_2\begin{pmatrix}
			y&z^{\frac{d+1}{3}}&x^2\\
			x&y^{\frac{2d-1}{3}}&z
		\end{pmatrix}$  and $\rmF(H)=\dfrac{2}{3}d^2-\dfrac{2}{3}$.
	\end{itemize} 
\end{lem}
\begin{proof}
	The argument is similar as in the proof of Theorem \ref{mainthm}.
\end{proof} 
\begin{cor}\label{CMofGr} Let $a,d$ be positive integer, $2\le d<a$ and $\gcd(a,d)=1$. Let $H=\left<a,a+1,a+d\right>$ is a numerical semigroup and $R=k[[H]]$ is numerical semigroup ring of $H$. Then if $q+r\ge d-2$ then the associated graded ring $\gr_\m(R)$ is Cohen-Macaulay.
\end{cor}
\begin{proof}
	By Theorem \ref{mainthm},  the defining ideal of $R=k[[H]]$ is given by
	$$I=\rmI_2\left( {\begin{array}{*{20}{c}}
			y^r&z&x^{q + r + 2 - d}\\ 
			x^{d - 1}&y^{d - r}&z^{q} 
	\end{array}} \right).$$
	This implies that $\beta+\beta'=d=\alpha+\gamma'$. Hence, by Herzog's criterion (see \cite{He81}) it follows that $\gr_\m(R)$ is Cohen-Macaulay.
\end{proof}
We close this section with a note that if $q+r<d-2$ then the associated graded ring $\gr_\m(R)$ may be not Cohen-Macaulay. For example, if $H=\left<d+1,d+2,2d+1\right>$ with $d\ge 5$ and assume that $H$ is not symmetric then $\gr_\m(R)$ is not Cohen-Macaulay. Indeed, by Lemma \ref{remaining}, the defining ideal $I$ of $R$ is either
$$\rmI_2\begin{pmatrix}
	y&z^{\frac{d}{3}}&x\\
	x^2&y^{\frac{2d}{3}}&z
\end{pmatrix} \text{ or }
\rmI_2\begin{pmatrix}
	y&z^{\frac{d+1}{3}}&x^2\\
	x&y^{\frac{2d-1}{3}}&z
\end{pmatrix}$$
which implies that $\beta+\beta'$ is greater than $\alpha+\gamma'$. This follows that $\gr_\m(R)$ is not Cohen-Macaulay.
\section{First Hilbert coefficients of numerical semigroup rings}

Let $H=\left<n_1,n_2,\ldots,n_e\right>$ be a numerical semigroup and $R=k[[H]]$ its numerical semigroup ring. Denote by $e_1(R)$ the first Hilbert coefficient of  $R$. In this section, we would like to find the formula for $e_1(R)$. For this purpose, we put $H'=\{c_1n_1+c_2(n_2-n_1)+c_3(n_3-n_1)+\cdots c_n(n_e-n_1)\mid  c_i\in\Z_{\ge 0},\, \forall i=1,2,\ldots,e\}$. Note that since $\gcd(n_1,n_2-n_1,n_3-n_1,...,n_e-n_1)=1$, $H'$ is a numerical semigroup generated by $n_1,n_2-n_1,n_3-n_1,...,n_e-n_1$. We call $H'$ the {\it blow-up} of $H$. The numerical semigroup ring $k[[H']]$ corresponds to the blow-up of $k[[H]]$.  
Let us recall the following result regarding the genus of $H$.
\begin{lem}[{\cite[Proposition 2.12]{RGS09}}]\label{genus}
	Let $a$ be a nonzero
	element of $H$. Then the genus of $H$ is given by
	$$g(H) =\frac{1}{a}\left(\sum_{w\in \Ap(H,a)} w\right)- \dfrac{(a-1)}{2}.$$
\end{lem}
The first Hilbert coefficient of  $R$ is deeply related to the genus of $H$ and its blow-up.
\begin{prop}\label{e1} There is an equality
	\begin{align*}
		e_1(R)=g(H)-g(H'),
	\end{align*}
	where $g(H), g(H')$ are genus of $H$ and $H'$, respectively.
\end{prop}
\begin{proof}
	We may assume the residue $k$ is infinite. Then $Q=(t^{n_1})$ is a minimal reduction of $\m$. Let $r = \text{red}_Q(\m):=\min\{n\in \Z\mid \m^{n+1} = Q\m^n\}$ and $S:= R[\frac{\m}{(t^{n_1})}]$ in the total quotient ring $Q(R)$ of $R$. The fact that (see \cite{GMP13}) we have $S=\frac{\m^r}{(t^{rn_1})}$, $R\subseteq S\subseteq \overline R$ and $e_1=\ell_R(S/R)$, where $\overline R$ is the integral closure of $R$ in $Q(R)$. Observe that 
	\begin{align*}
		S&=R[\frac{\m}{(t^{n_1})}]=R[\frac{(t^{n_1},t^{n_2},\ldots,t^{n_e})}{(t^{n_1})}]=R[t^{n_2-n_1},t^{n_3-n_1},\ldots,t^{n_e-n_1}]\\
		&=k[[t^{n_1},\ldots,t^{n_e}, t^{n_2-n_1},t^{n_3-n_1},\ldots,t^{n_e-n_1}]]\\
		&=k[[t^{n_1},t^{n_2-n_1},t^{n_3-n_1},\ldots,t^{n_e-n_1}]].
	\end{align*}
	Hence,
	\begin{align*}
		e_1(R)&=\ell_R(S/R)=\dim_kk[[H']]/k[[H]]\\
		&=\dim_kk[[t]]/k[[H]]-\dim_kk[[t]]/k[[H']]\\
		&=g(H)-g(H').
	\end{align*}
\end{proof}
The Proposition \ref{e1} allows us to calculate the first Hilbert coefficient for some special numerical semigroups.
\begin{ex} Let  $H=\left<a,b\right>$ be a numerial semigroup with $a<b$. Then by Proposition \ref{e1} and \cite[Proposition 2.13]{RGS09}, we have
	\begin{align*}
		e_1(R)&=g(\left<a,b\right>)-g(\left<a,b-a\right>)\\
		&=\dfrac{ab-a-b+1}{2}-\dfrac{a(b-a)-a-(b-a)+1}{2}\\
		&=\dfrac{a(a-1)}{2}.
	\end{align*}
\end{ex}

\begin{ex}
	Let $a,d,n$ be positive integer such that $\gcd(a,d)=1$. Let $H=\left<a,a+d,\ldots,a+nd\right>$ be the numerical semigroup generated by an arithmetic sequence. The blow up of $H$ is $H'=\left<a,d\right>$ so that $g(H')=\dfrac{ad-a-d+1}{2}$. Let $R=k[[H]]$. We write $a = qn + r$ with $0\le r<n$. We define the subsets
	$A_i$ of $H$ for all $1\le i\le q$ as the following
	$$A_i:=\{ia + jd\mid  (i-1)n + 1\le j\le in\}.$$
	Then, thanks to \cite{Ma04}  we have
	\begin{enumerate}
		\item[(i)] if $r = 0$, then $\Ap(H, a) = \{0\}\cup A_1\cup A_2\cup\ldots\cup A_{q-1}\cup (A_q\setminus\{qa + qnd\}).$
		\item [(ii)] if $r = 1$, then
		$\Ap(H, a) = \{0\}\cup A_1\cup A_2\cup\ldots\cup A_{q-1}\cup A_q.$
		\item [(iii)] if $2\le r\le n$, then
		$\Ap(H, a) = \{0\}\cup A_1\cup A_2\cup\ldots\cup A_q\cup \{(q+ 1)a + jd \mid qn + 1\le j\le qn + r-1\}.$
	\end{enumerate}
	Hence, by Lemma \ref{genus}, we get that
	\begin{enumerate}
		\item[(1)] If $r = 0$ or $r=1$, then $e_1(R)=\dfrac{nq(q+1)}{2}$.
		\item [(2)] If $2\le r\le n$, then
		$e_1(R)=\dfrac{nq(q+1)}{2}+\dfrac{r-1}{2}(2a(q+1)+2qn+r)$.
	\end{enumerate}
\end{ex}
\begin{rem}\label{remark1}
	For each $0\le i\le n_1-1$, we let $\omega_i=\min\{h\in H\mid h\equiv i\pmod {n_1}\}$ and $\omega'_i=\min\{h\in H'\mid h\equiv i\pmod {n_1}\}$. Then it is clear that $\omega_0=\omega'_0=0$ and $\Ap(H,n_1)=\{\omega_0,\omega_1,\ldots,\omega_{n_1-1}\}$,  $\Ap(H',n_1)=\{\omega'_0,\omega'_1,\ldots,\omega'_{n_1-1}\}$. For each $i>0$ we write $\omega_i=k_in_1+i, \omega'_i=\ell_in_1+i$ for some positive integers $k_i,\ell_i$. Firstly, we see that $k_i-\ell_i\ge 1$ for all $1\le i\le n_1-1$. Indeed, for every $1\le i\le n_1-1$, let $\omega_i=c_1n_1+\cdots+c_en_e$ with $c_j\ge 0$. Since $\omega_i>0$, there exists $c_j\ge 1$. Hence $$\omega_i-n_1=(c_1+\cdots+c_e-1)n_1+c_2(n_2-n_1)+\cdots+c_e(n_e-n_1)\in H'.$$ Note that $\omega_i-n_1\equiv i\pmod {n_1}$. This implies that $\omega_i-n_1\ge \omega'_i$. Therefore $k_i-\ell_i\ge 1$.
	Now we observe that, by Lemma \ref{genus}, $g(H)=\sum_{i=1}^{n_1-1}k_i$ and $g(H')=\sum_{i=1}^{n_1-1}\ell_i$ which implies that $$e_1(R)=\sum_{i=1}^{n_1-1}(k_i-\ell_i).$$
\end{rem}
The equality in Remark \ref{remark1} can be used effectively to calculate $e_1(R)$ in the case where numerical semigroups has maximal embedding dimension. 
\begin{lem}
	Let $R=k[[H]]$, where $H=\left<n_1,n_2,\ldots,n_e\right>$  be a numerical semigroup of maximal embedding dimension, that is, $n_1=e$. Then $e_1=n_1-1$.
\end{lem}
\begin{proof}
	For each $1\le i\le n_1-1$ we write $\omega'_i=c'_1n_1+c'_2(n_2-n_1)\cdots+c'_e(n_e-n_1)$ with $c'_j\ge 0$. Since $\omega'_i>0$ and $\omega'_i\not\equiv0\pmod {n_1}$, there exists $c'_j\ge 1$ with $j\ge 2$. We may assume $\omega'_i=c'_1n_1+c'_2(n_2-n_1)\cdots+c'_k(n_k-n_1)$ with $c'_j\ge 1$ for all $2\le k\le e$. Then $(c'_1+1)n_1+c'_2(n_2-n_1)=[((c'_1+1)n_1+n_2-n_1)+n_2-n_1]+\cdots+n_2-n_1\in H$. Here note that since $H$ has maximal embedding dimension, $x+y-n_1\in H$ whenever $x,y\in H\setminus\{0\}$. Similarly, we easily get that $\omega'_i+n_1\in H$ for all $2\le k\le e$. This follows that $\omega'_i+n_1\ge \omega_i$, whence $k_i-\ell_i\le 1$. Combining Remark \ref{remark1} we obtain that $k_i-\ell_i= 1$. This concludes that $$e_1=\sum_{i=1}^{n_1-1}1=n_1-1.$$
\end{proof}


Now, let $a,d$ be positive integer and assume that $2\le d<a$. Let $n_1=a, n_2=a+1, n_3=a+d$ and $H=\left<a,a+1,a+d\right>$ be a numerical semigroup. Let $R=k[[H]]$ is the numerical semigroup ring associated to $H$ and $e_1(R)$ is the first Hilbert coefficient of $R$. Let $S = k[[x,y,z]]$ be the formal series ring and $\varphi :S \to R$ be a $k$-algebra homomorphism given by  $\varphi(x)=t^{n_1}, \varphi(y)=t^{n_2}, \varphi(z)=t^{n_3}$. We write $a=dq+r\text{ with } 0\le r\le d-1.$  
\begin{prop}\label{e1first}
 Suppose  $a\ge d^2-3d$. Then \[
e_1(R)=\frac{a(a-1)}{2}
-(d-1)\left(\frac{d\,q(q-1)}{2}+rq\right).\]
\end{prop}
\begin{proof}
By the assumption $a\ge d^2-3d$ and Proposition \ref{prop37}, we have for each $i=0,1,\dots,a-1$ that
	\[
	\omega_i = i(a+1) - (d-1)a\left\lfloor \frac{i}{d}\right\rfloor.
	\]
Hence, by Lemma \ref{genus} we get
	\[
	g(H)=\frac{1}{a}\sum_{i=0}^{a-1}\omega_i-\frac{a-1}{2}.
	\]
Observe that
	\[
	\sum_{i=0}^{a-1}\omega_i
	=(a+1)\sum_{i=0}^{a-1} i
	-(d-1)a\sum_{i=0}^{a-1}\left\lfloor \frac{i}{d}\right\rfloor=(a+1)\frac{a(a-1)}{2}-(d-1)a\sum_{i=0}^{a-1}\left\lfloor \frac{i}{d}\right\rfloor.
	\]

On the other hands, since $a = qd + r$, each integer $k=0,1,\dots,q-1$ occurs exactly $d$ times among the values
	$\lfloor i/d\rfloor$ for $i=0,\dots,a-1$, and the value $q$ occurs $r$ times.
	Hence,
	\[
	\sum_{i=0}^{a-1}\left\lfloor \frac{i}{d}\right\rfloor
	= d\sum_{k=0}^{q-1}k + rq
	= d\frac{q(q-1)}{2}+rq.
	\]

We obtain
	\[
	\sum_{i=0}^{a-1}\omega_i
	=(a+1)\frac{a(a-1)}{2}
	-(d-1)a\left(d\frac{q(q-1)}{2}+rq\right).
	\]
Hence, 
	\begin{align*}
e_1(R)=g(H)=\frac{a(a-1)}{2}
		-(d-1)\left(d\frac{q(q-1)}{2}+rq\right).
	\end{align*}

\end{proof}

\begin{prop}\label{Fro-Chern} Suppose that $q+r\ge d-2$ and $H$ is not symmetric. Then the following assertions hold true.
	\begin{itemize}
		\item[1)] If $1\le r\le d-2$ then $e_1(R)=\dfrac{1}{2d}(a^2+d(d-2)a-r(d-1)(d-r)$.
		\item[2)] If $r=d-1$ then $e_1(R)=\dfrac{1}{2d}(a^2+d(d-2)a-(d-1)^2)$.
	\end{itemize}
\end{prop}
\begin{proof} Note that since $H=\left<a,a+1,a+d\right>$, the blow-up $H'=\N$. Hence the genus of $H'$ is $g(H')=0$. It implies that $e_1(R)=g(H)$ by Proposition \ref{e1}. Now, thanks to Theorem \ref{mainthm}, the parameters in $(\sharp)$ are given by
	$$\alpha=d-1,\beta=d-r,\gamma=i,\alpha'=q+r+2-d,\beta'=r, \text{ and } \gamma'=1.$$
	\begin{itemize}
		\item If $r=d-1$ then $\beta'(a+1)=r(a+1)>(d-1)a=\alpha a$. 
		In this case by \cite{NNW12}, we obtain $$g(H)=\dfrac{\alpha\beta\gamma+\rmF(H)+1}{2}=\frac{1}{2d}(a^2+d(d-2)a-(d-1)^2).$$
		\item If $1\le r\le d-2$ then $\beta'(a+1)=r(a+1)<(d-1)a=\alpha a$. In this case, again by by \cite{NNW12}, we have $$g(H)=\dfrac{\alpha'\beta'\gamma'+\rmF(H)+1}{2}=\frac{1}{2d}(a^2+d(d-2)a-r(d-1)(d-r).$$
	\end{itemize}
	These implies the desired conclusion.
\end{proof}
\begin{rem} If $d\ge 5$ and $a=qd+r$ with $1\le q+r\le d-3$  we have not yet known the general formula for $e_1(R)$. We have the following in the case $q=r=1$. Let $H=\left<d+1,d+2,2d+1\right>$ with $d\ge 5$. Then, by Lemma \ref{remaining} we have
	\begin{itemize}
		\item if $d\equiv 0\pmod 3$ then and $e_1(R)=\dfrac{1}{3}d^2+\dfrac{1}{3}d.$ 
		\item if $d\equiv 1\pmod 3$ then $e_1(R)=\dfrac{1}{3}d^2+\dfrac{1}{3}d+\dfrac{1}{3}.$ 
		\item if $d\equiv 2\pmod 3$ then $e_1(R)=\dfrac{1}{3}d^2+\dfrac{1}{3}d.$ 
	\end{itemize} 
\end{rem}
\section*{Acknowledgments}
We are grateful to Duong Thi Huong, Nguyen Tuan Long and Vu Thi Minh Phuong for their valuable discussions.

\end{document}